\documentclass{amsart}

\usepackage{enumerate}    % To support \begin{enumerate}[(i)]
\usepackage{graphicx}     % For embedding figures

\newtheorem{theorem}{Theorem}[section]
\newtheorem{lemma}[theorem]{Lemma}
\newtheorem{corollary}[theorem]{Corollary}

\newtheorem{observation}[theorem]{Observation}

\theoremstyle{definition}
\newtheorem{definition}[theorem]{Definition}
\newtheorem*{boldremark}{Remark}

\newcommand{\auxgraph}{\Gamma}

\newcommand{\co}{\colon\thinspace}
\newcommand{\facets}{\Phi}
\newcommand{\jm}{\jmof{P}}
\newcommand{\jmof}[1]{\mathcal{J}_{#1}}

\newcommand{\R}{\mathbb{R}}
\newcommand{\rank}{\mathop{\mathrm{rank}}}

\newcommand{\Z}{\mathbb{Z}}

\begin{document}

%%%%%%%%%%%%%%%%%%%%%%%%%%%%%%%%%%%%%%%%%%%%%%%%%%%%%%%%%%%%%%%%%%%%%%%%
%
%   Title and abstract
%
%%%%%%%%%%%%%%%%%%%%%%%%%%%%%%%%%%%%%%%%%%%%%%%%%%%%%%%%%%%%%%%%%%%%%%%%

\title[Complementary vertices and adjacency testing in polytopes]%
    {Complementary vertices and \\ adjacency testing in polytopes}
\author{Benjamin A.\ Burton}
\address{School of Mathematics and Physics \\
    The University of Queensland \\
    Brisbane QLD 4072 \\
    Australia}
\email{bab@maths.uq.edu.au}
\thanks{The author is supported by the Australian Research Council
    under the Discovery Projects funding scheme (project DP1094516).}
 % Convex and discrete geometry; polytopes and polyhedra;
           % Computational aspects related to convexity
\begin{abstract}
    Our main theoretical result is that,
    if a simple polytope has a pair of complementary vertices
    (i.e., two vertices with no facets in common),
    then it has at least two such pairs, which can be chosen to be disjoint.
    Using this result,
    we improve adjacency testing for vertices in both
    simple and non-simple polytopes:
    given a polytope in the standard form
    $\{\mathbf{x}\in\mathbb{R}^n\,|\,A\mathbf{x}=\mathbf{b}\ \mbox{and}\
    \mathbf{x}\geq 0\}$ and a list of its $V$ vertices,
    we describe an $O(n)$ test to identify whether any two given
    vertices are adjacent.  For simple polytopes this test is perfect;
    for non-simple polytopes it may be indeterminate, and instead acts
    as a filter to identify non-adjacent pairs.
    Our test requires an \mbox{$O(n^2V+nV^2)$} precomputation,
    which is acceptable in settings such as all-pairs adjacency testing.
    These results improve upon the more general $O(nV)$ combinatorial
    and $O(n^3)$ algebraic adjacency tests from the literature.

    \bigskip

    \noindent
    \textsc{Conference vs journal versions.}
    This is the journal version of a paper that appeared in
    \textit{Computing and Combinatorics: 18th Annual International
    Conference, COCOON 2012},
    Lecture Notes in Comput.\ Sci., vol.~7434, Springer, 2012, pp.~507--518.
    This journal version restructures and extends Section~\ref{s-comp}.
    It strengthens the main result (Theorem~\ref{t-comp})
    to show that the two pairs of complementary vertices can be made
    disjoint, provides richer supporting information on facet sets, and
    incorporates additional results for the case of a $d$-polytope
    with $2d$ facets.
\end{abstract}

\maketitle

%%%%%%%%%%%%%%%%%%%%%%%%%%%%%%%%%%%%%%%%%%%%%%%%%%%%%%%%%%%%%%%%%%%%%%%%
%
%   Introduction
%
%%%%%%%%%%%%%%%%%%%%%%%%%%%%%%%%%%%%%%%%%%%%%%%%%%%%%%%%%%%%%%%%%%%%%%%%

\section{Introduction}

Two vertices of a polytope are \emph{complementary} if they do not
belong to a common facet.  Complementary vertices play an important role
in the theory of polytopes; for instance, they provide the setting for
the $d$-step conjecture \cite{kim10-hirsch,klee67-dstep}
(now recently disproved \cite{santos12-hirsch}),
and in the dual setting of disjoint facets
they play a role in the classification of compact hyperbolic Coxeter polytopes
\cite{felikson09-coxeter}.
In game theory, Nash equilibria of bimatrix games are described by an
analogous concept of complementary vertices in \emph{pairs} of polytopes
\cite{jansen81-nash,winkels79-bimatrix}.
% See \cite{vonstengel99-maximal} for a study of the maximal number of
% complementary vertices in various settings.

Our first main contribution, presented in Section~\ref{s-comp},
relates to the minimal
number of complementary vertex pairs.  Many polytopes have
no pairs of complementary vertices at all (for instance,
any neighbourly polytope).  However, we prove here that
if a simple polytope $P$ of dimension $d>1$
has at least one pair of complementary
vertices, then it must have at least \emph{two} such pairs.
Moreover, these two pairs can be chosen to have no vertex in common.

The proof involves the construction of paths
through a graph whose nodes represent pairs of complementary or
``almost complementary'' vertices of $P$.  In this sense it is
reminiscent of the Lemke-Howson algorithm for constructing
Nash equilibria in bimatrix games \cite{lemke64-games},
although our proof operates in a less well-controlled setting.
We discuss this relationship further in Section~\ref{s-disc}.

Our second main contribution, presented in Section~\ref{s-adj},
is algorithmic: we use our first theorem to build a fast adjacency test.
Specifically, given a polytope in the standard form
$P = \{\mathbf{x}\in\mathbb{R}^n\,|\,A\mathbf{x}=\mathbf{b}\ \mbox{and}\
\mathbf{x}\geq 0\}$ with $V$ vertices,
we begin with an $O(n^2V+nV^2)$ time precomputation step,
after which we can test any two vertices for adjacency in $O(n)$ time.
If $P$ is simple (which the algorithm can also identify)
then this test always gives a precise response;
otherwise it may be indeterminate but it
can still assist in identifying non-adjacent pairs.

The key idea is, for each pair of vertices $u,v \in P$, to compute the
join $u \vee v$; that is, the minimal face containing both $u$ and $v$.
If $P$ is simple then our theorem on complementary vertices
shows that $u \vee v = u' \vee v'$ for a second pair of vertices $u',v'$.
Our algorithm then identifies such ``duplicate'' joins.

Although the precomputation is significant, if we are testing all
$\binom{V}{2}$ pairs of vertices for adjacency then it does not increase
the overall time complexity.  Our $O(n)$ test then becomes
extremely fast, outperforming the standard $O(nV)$ combinatorial
and $O(n^3)$ algebraic tests from the literature \cite{fukuda96-doubledesc}.
Even in the non-simple setting, our test can be used as a fast
pre-filter to identify non-adjacent pairs of vertices,
before running the more expensive standard tests on those pairs that remain.

In Section~\ref{s-disc} we discuss these performance issues further,
as well as the application of these ideas to the key problem of
polytope vertex enumeration.

All time complexities are measured using the arithmetic model of computation,
where we treat each arithmetical operation as constant-time.

We briefly remind the reader of the necessary terminology.
Following Ziegler \cite{ziegler95}, we insist that all polytopes be bounded.
A \emph{facet} of a $d$-dimensional polytope $P$ is a
$(d-1)$-dimensional face, and two vertices of $P$ are \emph{adjacent}
if they are joined by an edge.
$P$ is \emph{simple} if every vertex belongs to precisely
$d$ facets (i.e., every vertex figure is a $(d-1)$-simplex),
and $P$ is \emph{simplicial} if every facet contains precisely $d$ vertices
(i.e., every facet is a $(d-1)$-simplex).
As before, two vertices of $P$ are \emph{complementary} if they do not
belong to a common facet; similarly, two facets
of $P$ are \emph{disjoint} if they do not contain a common vertex.

%%%%%%%%%%%%%%%%%%%%%%%%%%%%%%%%%%%%%%%%%%%%%%%%%%%%%%%%%%%%%%%%%%%%%%%%
%
%   Complementary vertices
%
%%%%%%%%%%%%%%%%%%%%%%%%%%%%%%%%%%%%%%%%%%%%%%%%%%%%%%%%%%%%%%%%%%%%%%%%

\section{Complementary Vertices} \label{s-comp}

In this section we prove the main theoretical result of this paper:

\begin{theorem} \label{t-comp}
    Let $P$ be a simple polytope of dimension $d > 1$.
    If $P$ has a pair of complementary vertices,
    then $P$ has at least two disjoint pairs of complementary vertices.
\end{theorem}

By ``disjoint'', we mean that
these pairs cannot be of the form $\{u,v\}$ and $\{u,w\}$;
instead they must involve four distinct vertices of $P$.

The proof of Theorem~\ref{t-comp} involves an
auxiliary graph, which we now describe.
To avoid confusion with vertices and edges
of polytopes, we describe graphs in terms of \emph{nodes} and
\emph{arcs}.  We do not allow graphs to have loops (i.e., arcs that
join a node to itself), or multiple arcs (i.e., two or more arcs
that join the same pair of nodes).

\begin{definition}[Auxiliary graph] \label{d-auxgraph}
    Let $P$ be a simple polytope of dimension $d > 1$.
    We construct the \emph{auxiliary graph} $\auxgraph(P)$ as follows:
    \begin{itemize}
        \item The \emph{nodes} of $\auxgraph(P)$ are unordered pairs of
        vertices $\{u,v\}$ of $P$ where $u,v$ have at most one facet in
        common.
        % (since $d>1$ it follows that $u$ and $v$ must be distinct).
        We say the node $\{u,v\}$ is of \emph{type~A} if
        the vertices $u,v \in P$ are complementary (they have no facets
        in common), or of \emph{type~B} otherwise (they have
        precisely one facet in common).

        \smallskip

        \item The \emph{arcs} of $\auxgraph(P)$ join nodes of the form
        $\{u,x\}$ and $\{u,y\}$, where $x$ and $y$ are adjacent vertices
        of $P$, and where
        no single facet of $P$ contains all three vertices $u,x,y$.
        For each arc $\alpha$, we define the \emph{facet set}
        $\facets(\alpha)$ to be the set of all facets of $P$
        that contain either the vertex $u$ or the edge $xy$.
    \end{itemize}
\end{definition}

\begin{figure}[htb]
\centering
\includegraphics[scale=0.7]{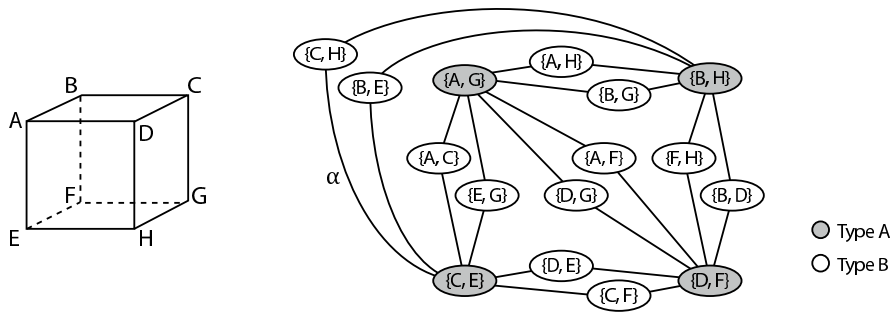}
\caption{A polytope $P$ and the corresponding auxiliary graph $\auxgraph(P)$}
\label{fig-cube}
\end{figure}

Figure~\ref{fig-cube} illustrates this graph for the case where $P$ is a cube.
Informally, each arc of $\auxgraph(P)$ modifies a node by ``moving''
one of its two vertices along an edge of $P$, so that if the
two vertices lie on a common facet $F$ then this movement is away from $F$.

To illustrate facet sets, consider the arc $\alpha$ joining $\{C,H\}$
with $\{C,E\}$.  The corresponding facet set $\facets(\alpha)$
has size $|\facets(\alpha)|=5$, and contains every facet of the
cube except for the left facet $\mathit{ABFE}$ (since every other facet
contains either the vertex $C$ or the edge $EH$).

More generally, it is clear that each facet set $\facets(\alpha)$ has
size $|\facets(\alpha)|=2d-1$, since by definition of $\auxgraph(P)$,
if an arc $\alpha$ joins nodes $\{u,x\}$ and $\{u,y\}$ then
the $d$ facets that contain
$u$ and the $d-1$ facets that contain the edge $xy$ must all be distinct.

We can formally characterise the arcs of $\auxgraph(P)$ as follows:

\begin{lemma} \label{l-arcs}
    Let $\nu=\{u,v\}$ be a node of $\auxgraph(P)$ as outlined above.
    \begin{enumerate}[(i)]
        \item If $\nu$ is of type~A, there are precisely $2d$ arcs
        $\alpha_1,\ldots,\alpha_{2d}$ meeting $\nu$.
        These include $d$ arcs
        %$\alpha_1,\ldots,\alpha_d$
        that connect
        $\nu$ with $\{u,x\}$ for every vertex $x$ adjacent to $v$ in $P$,
        and $d$ arcs
        %$\alpha_{d+1},\ldots,\alpha_{2d}$
        that connect $\nu$ with $\{y,v\}$ for
        every vertex $y$ adjacent to $u$ in $P$.
        The $2d$ facet sets
        $\facets(\alpha_1),\ldots,\facets(\alpha_{2d})$ are all distinct,
        and each consists of all but one of the $2d$ facets that
        touches either $u$ or $v$.

        \smallskip

        \item If $\nu$ is of type~B, there are precisely two arcs
        $\alpha,\alpha'$
        meeting $\nu$.  Let $F$ be the unique facet of $P$ containing both
        $u$ and $v$.  Then these two arcs join $\nu$ with
        $\{u,x\}$ and $\{y,v\}$, where
        $x$ is the unique vertex of $P$ adjacent to $v$
        for which $x \notin F$, and $y$ is the unique vertex of $P$ adjacent
        to $u$ for which $y \notin F$.
        The two facet sets $\facets(\alpha),\facets(\alpha')$ are identical,
        and each consists of all $2d-1$ facets that touch either $u$ or $v$.
    \end{enumerate}
\end{lemma}

\begin{proof}
    We consider the type~A and B cases in turn.
    \begin{enumerate}[(i)]
        \item Let $\nu=\{u,v\}$ be of type~A, and let $x$ be any vertex
        of $P$ adjacent to $v$.
        Since $\nu$ is of type~A, $u$ and $v$ have no facets in common.
        Since $P$ is simple, exactly one facet of $P$ contains $x$ but not $v$.
        Therefore $u$ and $x$ have at most one facet in common,
        and so $\{u,x\}$ is a node of $\auxgraph(P)$.
        Moreover, since no facet contains
        both $u$ and $v$, the nodes $\nu=\{u,v\}$ and $\{u,x\}$
        are joined by an arc.
        Because $P$ is simple there are precisely $d$ vertices adjacent
        to $v$, yielding precisely $d$ arcs of this type.

        A similar argument shows that there are another $d$ arcs
        that join $\nu$ to nodes $\{y,v\}$ where $y$ is any vertex of
        $P$ adjacent to $u$, yielding a total of $2d$ arcs that meet
        $\nu$ overall.

        Consider any arc $\alpha_i$ that joins $\{u,v\}$ with some node
        $\{u,x\}$, where $v$ and $x$ are adjacent vertices as described above.
        Because $u$ and $v$ are complementary, the $d$ facets that touch
        $u$ and the $d$ facets that touch $v$ are all distinct,
        and it follows that the facet set
        $\facets(\alpha_i)$ contains all $2d$ of
        these facets except for the unique facet $F_{v,x}$ that touches
        $v$ but not the adjacent vertex $x$.

        As we cycle through all $d$ possibilities for the adjacent
        vertex $x$, this
        ``missing'' facet $F_{v,x}$ cycles through all $d$ facets that
        touch $v$.  Likewise, when we consider arcs $\alpha_i$ that join
        $\{u,v\}$ with some node $\{y,v\}$, the ``missing'' facet cycles
        through all $d$ facets that touch $u$.
        Since these $2d$ ``missing'' facets are all distinct,
        it follows that the $2d$ facet sets
        $\facets(\alpha_1),\ldots,\facets(\alpha_{2d})$ are likewise
        distinct.

        \smallskip

        \item Now let $\nu=\{u,v\}$ be of type~B, and let $F$ be the
        unique facet of $P$ containing both $u$ and $v$.  If there is
        an arc from $\nu$ to any node of the form $\{u,x\}$,
        it is clear from Definition~\ref{d-auxgraph} that we must have
        $x$ adjacent to $v$ and $x \notin F$.  Because $P$ is simple,
        there is precisely one $x$ with these properties.

        We now show that an arc from $\nu$ to $\{u,x\}$ does indeed exist.
        Because $P$ is simple, at most one facet of $P$ contains
        $x$ but not $v$.  The vertex $v$ in turn has only the facet $F$
        in common with $u$; since $x \notin F$ it follows that
        $u$ and $x$ have at most one facet in common.  Therefore $\{u,x\}$ is
        a node of $\auxgraph(P)$.  Because $x \notin F$ there is
        no facet containing all of $u$, $v$ and $x$, and so the arc from $\nu$
        to $\{u,x\}$ exists.

        A similar argument applies to the arc from $\nu$ to $\{y,v\}$,
        yielding precisely two arcs that meet $\nu$ as described in the
        lemma statement.

        Consider the arc $\alpha$ that joins $\{u,v\}$ with $\{u,x\}$
        as described above.
        The only facet of $P$ that touches $v$ but does not
        contain the edge $vx$ is the common facet $F$, which nevertheless
        touches the vertex $u$.
        Therefore the facet set
        $\facets(\alpha)$ is precisely the set of all facets that touch
        either $u$ or $v$.  The same is true for the second arc
        $\alpha'$ that joins $\{u,v\}$ with $\{y,v\}$,
        whereby we obtain $\facets(\alpha')=\facets(\alpha)$.
    \end{enumerate}

    To finish, we note that every vertex belongs to $d>1$ facets,
    and so every node $\{u,v\}$ of $\auxgraph(P)$ has $u \neq v$.
    This ensures that we do not double-count arcs in our argument;
    that is, the $2d$ arcs in case~(i) join $\nu$ to $2d$ distinct nodes of
    $\auxgraph(P)$, and likewise the two arcs in case~(ii) join $\nu$
    to two distinct nodes of $\auxgraph(P)$.
\end{proof}

Our overall strategy for proving Theorem~\ref{t-comp} is to show that,
if we follow any path from a type~A node of $\auxgraph(P)$,
we must arrive at some \emph{different} type~A node; that is,
we obtain a new pair of complementary vertices.
Furthermore, we show that some such path has one or more
intermediate type~B nodes,
and as a result the type~A nodes that it connects must
represent \emph{disjoint} complementary pairs.
The details are as follows.

\begin{proof}[Proof of Theorem~\ref{t-comp}]
    To establish Theorem~\ref{t-comp}, we must prove that if
    $\auxgraph(P)$ contains at least one type~A node then it contains at
    least two type~A nodes, and that moreover we can find
    two type~A nodes that represent four distinct vertices of $P$.

    Let $\nu$ be a type~A node, and let $\alpha$ be any arc
    meeting $\nu$.  Since every type~B node has degree~two
    (by Lemma~\ref{l-arcs}), this arc
    $\alpha$ begins a well-defined path through $\auxgraph(P)$ that passes
    through zero or more type~B nodes in sequence, until either
    (a)~it arrives at a new type~A node $\nu'$, or (b)~it returns to the
    original type~A node $\nu$ and becomes a cycle
    (see Figure~\ref{fig-paths}).
    Our first task is to prove that case~(b) is impossible.

    \begin{figure}[htb]
    \centering
    \includegraphics{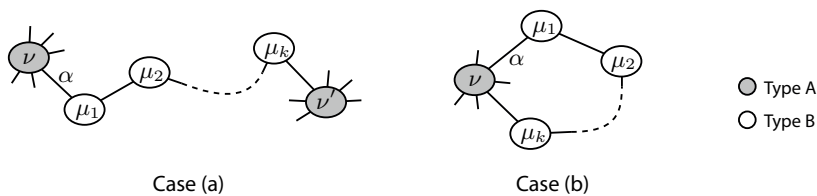}
    \caption{Following a path from the type~A node $\nu$}
    \label{fig-paths}
    \end{figure}

    Suppose then that case~(b) occurs, and there is a cycle that 
    passes through nodes $\nu,\mu_1,\ldots,\mu_k,\nu$ in turn,
    where $\nu$ is of type~A and the intermediate nodes
    $\mu_1,\ldots,\mu_k$ are of type~B.
    By Lemma~\ref{l-arcs}, all arcs on this cycle have identical facet
    sets (since they are joined by type~B nodes).  However,
    Lemma~\ref{l-arcs} also shows that the first and last arcs
    on this cycle must have different facet sets
    (since they both meet the type~A node $\nu$).
    This yields a contradiction.

    Therefore we have a path that joins two distinct type~A nodes $\nu$
    and $\nu'$, which means that our polytope contains two distinct
    pairs of complementary vertices.  All that remains is to show that
    we can find two such pairs that are \emph{disjoint}, i.e., that
    together use four distinct vertices of $P$.

    We claim that, somewhere in the auxiliary graph $\auxgraph(P)$,
    there is at least one path of the form
    $\nu,\mu_1,\ldots,\mu_k,\nu'$ where the end nodes $\nu,\nu'$ are of
    type~A, the intermediate nodes $\mu_1,\ldots,\mu_k$ are of type~B,
    and the path contains \emph{at least one type~B node} (i.e., $k \geq 1$).
    To see this:
    let $\{u,v\}$ be a pair of complementary vertices, and let
    $u=z_1,z_2,\ldots,z_q=v$ be a path that follows edges of the
    polytope $P$ from vertex $u$ to $v$.
    Since $\{z_1,v\}$ is a complementary pair but $\{z_q,v\}$ is not,
    there must be some $i$ for which $\{z_i,v\}$ is a complementary
    pair but $\{z_{i+1},v\}$ is not.  It follows that,
    since $z_i$ and $z_{i+1}$ are adjacent in $P$, there must be a
    type~A node $\{z_i,v\}$ with an arc leading to the type~B node
    $\{z_{i+1},v\}$.  Following this arc yields a path of the form that
    we seek.

    We can therefore consider such a path
    $\nu,\mu_1,\ldots,\mu_k,\nu'$ in $\auxgraph(P)$ that joins two
    type~A nodes and contains $k \geq 1$ type~B nodes.
    We claim that the endpoints $\nu,\nu'$ of this path must be
    \emph{disjoint} pairs of complementary vertices.
    Suppose this is not true: since $\nu \neq \nu'$ we can assume
    that $\nu=\{u,v\}$ and $\nu'=\{u,w\}$ for three distinct vertices
    $u,v,w$ of $P$.
    As before, all arcs on this path must have identical facet sets
    (since they are joined by type~B nodes);
    let this common facet set be $\Phi=\{F_1,\ldots,F_{2d-1}\}$.  By applying
    Lemma~\ref{l-arcs} to the type~A endpoints, we must be in one of the
    following situations:
    \begin{enumerate}
        \item All $d$ facets that touch $u$ are in $\Phi$.
        Without loss of generality, suppose these facets are
        $F_1,\ldots,F_d$.  By Lemma~\ref{l-arcs} again,
        the remaining facets $F_{d+1},\ldots,F_{2d-1}$ touch both
        $v$ and $w$, and their intersection
        $F_{d+1} \cap \ldots \cap F_{2d-1}$ must therefore be a face of $P$
        that contains both $v$ and $w$.
        Because $P$ is simple this face has codimension $\geq d-1$;
        that is, $v$ and $w$ are joined by an \emph{edge},
        and are therefore adjacent in $P$.

        It follows that the nodes $\nu=\{u,v\}$ and $\nu'=\{u,w\}$ are
        joined in $\auxgraph(P)$ by a single arc $\alpha$, whose facet set
        must again be $\Phi=\{F_1,\ldots,F_{2d-1}\}$.
        This means that $\nu$ has two outgoing arcs with
        the same facet set $\Phi$ (one to $\mu_1$ and one to $\nu'$),
        contradicting Lemma~\ref{l-arcs} which states that the facet
        sets of arcs leaving a type~A node must all be distinct.

        \smallskip

        \item Only $d-1$ of the facets that touch $u$ are in $\Phi$.
        Without loss of generality, suppose these facets are
        $F_1,\ldots,F_{d-1}$.  By applying Lemma~\ref{l-arcs} to the type~A
        node $\nu=\{u,v\}$, it follows that the $d$ facets that meet
        $v$ must be $F_d,\ldots,F_{2d-1}$.  Likewise, applying
        Lemma~\ref{l-arcs} to the type~A node $\nu'=\{u,w\}$,
        we find that the $d$
        facets that meet $w$ must be $F_d,\ldots,F_{2d-1}$.
        Therefore $v$ and $w$ meet the same facets of $P$, contradicting
        the assumption that $v$ and $w$ are distinct vertices of $P$.
        \qedhere
    \end{enumerate}
\end{proof}

\begin{boldremark}
    The proof of Theorem~\ref{t-comp} is algorithmic: given a
    simple polytope $P$ of dimension $d>1$ and a pair of complementary
    vertices $u,v \in P$, it gives an explicit algorithm for
    locating a second pair of complementary vertices.

    In essence, we arbitrarily replace one of the vertices $u$
    with an adjacent vertex $u'$,
    and then repeatedly adjust this pair of vertices according to
    Lemma~\ref{l-arcs} part~(ii)
    until we once again reach a pair of complementary vertices.
    For each adjustment, Lemma~\ref{l-arcs} part~(ii) gives two options
    (corresponding to the two arcs that meet a type~B node);
    we always choose the option that leads us ``forwards'' to a new
    pair, and not ``backwards'' to the pair we had immediately before.

    The proof above ensures that we will eventually reach a complementary
    pair of vertices again, and that these will not be the same as the
    original pair $u,v$.  The proof also gives a (more complex)
    algorithmic procedure for obtaining two \emph{disjoint} pairs of
    complementary vertices; note that these might both be different from
    the pair that we started with.
\end{boldremark}

Passing to the dual polytope, Theorem~\ref{t-comp} gives us
an immediate corollary:

\begin{corollary} \label{c-disjoint}
    Let $P$ be a simplicial polytope of dimension $d>1$.
    If $P$ has a pair of disjoint facets,
    then $P$ has at least two disjoint pairs of disjoint facets.
\end{corollary}

Again, by ``disjoint pairs'' we mean that
these pairs cannot be of the form $\{F,G\}$ and $\{F,H\}$;
instead they must involve four distinct facets of $P$.
We do \emph{not} mean that all four facets are pairwise
disjoint (which in general need not be true).

We now observe that the ``simple'' and ``simplicial''
conditions are necessary in Theorem~\ref{t-comp} and
Corollary~\ref{c-disjoint}.

\begin{observation} \label{o-nonsimple}
    The triangular bipyramid (Figure~\ref{fig-counterex}, left) is a
    non-simple polytope of dimension $d=3$ with precisely one pair of
    complementary vertices (the apexes at the top and bottom, shaded
    in the diagram).

    Its dual is the triangular prism (Figure~\ref{fig-counterex}, right),
    which is a non-simplicial polytope of dimension $d=3$ with precisely
    one pair of disjoint facets (the triangles at the top and bottom,
    shaded in the diagram).
\end{observation}

\begin{figure}[htb]
\centering
\includegraphics[scale=0.6]{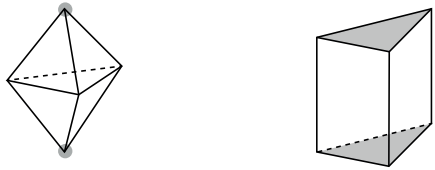}
\caption{A triangular bipyramid (left) and a triangular prism (right)}
\label{fig-counterex}
\end{figure}

These constructions are easily generalised.
For instance, we can build a non-simple bipyramid over a
neighbourly polytope: the two apexes form the unique pair
of complementary vertices, and all other pairs of vertices are adjacent.
Felikson and Tumarkin provide further examples in the dual setting
\cite{felikson09-coxeter},
involving non-simplicial Coxeter polytopes with precisely
one pair of disjoint facets.

We finish this section with a parity result for the case of a
$d$-dimensional polytope with precisely $2d$ facets.
This is a ``smallest case scenario'', in that any polytope with a pair of
complementary vertices must have at least $2d$ facets (since
$d$ facets must touch each vertex).
In this setting we can be more precise:

\begin{theorem} \label{t-2d}
    Let $P$ be a simple polytope of dimension $d > 1$
    with precisely $2d$ facets.
    Then the total number of pairs of complementary vertices of $P$
    is even.
    Moreover, all of these pairs are pairwise disjoint
    (i.e., no two pairs have a vertex in common).
\end{theorem}

\begin{proof}
    Denote the facets of $P$ by $F_1,\ldots,F_{2d}$,
    and let $\Phi^\ast$ denote the full set of facets
    $\Phi^\ast=\{F_1,\ldots,F_{2d}\}$.
    For each $i=1,\ldots,2d$,
    let $\Phi^{(i)}$ denote the $(2d-1)$-element subset
    $\Phi^{(i)} = \Phi^\ast \backslash \{F_i\}$.

    Let $\nu=\{u,v\}$ be any type~A node in the auxiliary graph
    $\auxgraph(P)$.  Since $u$ and $v$ are complementary, all $2d$ facets
    $F_1,\ldots,F_{2d}$ must touch either $u$ or $v$, and it follows from
    Lemma~\ref{l-arcs} that the $2d$ distinct facet sets for the $2d$ arcs
    that meet $\nu$ must be $\Phi^{(1)},\ldots,\Phi^{(2d)}$ in some order.

    In particular, for each type~A node $\nu$, \emph{exactly one}
    arc meeting $\nu$ has the facet set $\Phi^{(1)}$.  On the other
    hand, by Lemma~\ref{l-arcs} again, for each type~B node $\nu'$
    either both arcs meeting $\nu'$ have the facet set $\Phi^{(1)}$,
    or else neither do.
    Because each arc has two endpoints, it follows that the
    total number of type~A nodes must be even; that is,
    the total number of pairs of complementary vertices of $P$ is even.

    We now show that no two of these pairs have a vertex in common.
    Suppose $\{u,v\}$ and $\{u,w\}$ are both complementary
    pairs of vertices, and without loss of generality suppose that $u$
    touches the $d$ facets $F_1,\ldots,F_d$.  Then because the
    first pair is complementary, $v$ must touch the $d$ facets
    $F_{d+1},\ldots,F_{2d}$;
    likewise, because the second pair is complementary, $w$ must touch
    the $d$ facets $F_{d+1},\ldots,F_{2d}$.  Therefore both $v$ and $w$ touch
    the same facets of $P$, and must be the same vertex of $P$.
\end{proof}

Again, we obtain an immediate corollary in the dual setting:

\begin{corollary}
    Let $P$ be a simplicial polytope of dimension $d > 1$
    with precisely $2d$ vertices.
    Then the total number of pairs of disjoint facets of $P$
    is even.
    Moreover, all of these pairs are pairwise disjoint
    (i.e., no two pairs have a facet in common).
\end{corollary}

Theorem~\ref{t-2d} cannot be extended to an arbitrary number of facets,
since in the general case the total number of pairs of complementary
vertices could be either odd or even:

\begin{figure}[htb]
    \centering
    \includegraphics[scale=0.5]{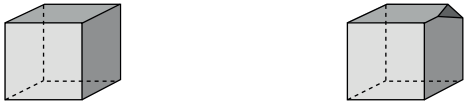}
    \caption{Even and odd numbers of complementary vertex pairs}
    \label{fig-parity}
\end{figure}

\begin{observation} \label{o-odd}
    The 3-dimensional cube (Figure~\ref{fig-parity}, left) is a simple
    polytope with a positive even number of pairs of
    complementary vertices.

    The 3-dimensional cube with one truncated vertex
    (Figure~\ref{fig-parity}, right) is a simple polytope
    with a positive odd number of pairs of complementary vertices.
\end{observation}

A careful count shows that these examples have four and nine pairs of
complementary vertices respectively.

%%%%%%%%%%%%%%%%%%%%%%%%%%%%%%%%%%%%%%%%%%%%%%%%%%%%%%%%%%%%%%%%%%%%%%%%
%
%   Adjacency testing
%
%%%%%%%%%%%%%%%%%%%%%%%%%%%%%%%%%%%%%%%%%%%%%%%%%%%%%%%%%%%%%%%%%%%%%%%%

\section{Adjacency Testing} \label{s-adj}

In this section we prove our main algorithmic result, which uses
Theorem~\ref{t-comp} to identify simple polytopes and
test for adjacent vertices.
Throughout this section we work with polytopes of the form
\begin{equation}
    P = \{\mathbf{x}\in\R^n\,|\,A\mathbf{x}=\mathbf{b}\ \mbox{and}\
    \mathbf{x}\geq 0\},
    \label{eqn-std}
\end{equation}
where $A$ is some $n$-column matrix.
This form is standard in mathematical programming,
and appears in key applications of vertex enumeration
\cite{benson98-outer,burton10-dd}.  Note that the dimension of $P$ is not
immediately clear from (\ref{eqn-std}),
although $n - \rank A$ gives an upper bound;
likewise, it is not immediately clear
whether $P$ is simple.

We recall some standard terminology from polytope theory:
if $F$ and $G$ are faces of the polytope $P$, then
the \emph{join} $F \vee G$ is the unique smallest-dimensional
face that contains both $F$ and $G$ as subfaces.
For example, recall the cube from Figure~\ref{fig-cube}.
For vertices $B$ and $C$, the join $B \vee C$ is the edge $\overline{BC}$,
whereas for vertices $A$ and $H$, the join $A \vee H$ is the square
facet $\mathit{ADHE}$ (the smallest face containing both vertices).
For edges $\overline{\mathit{AD}}$ and $\overline{\mathit{DC}}$,
the join $\overline{\mathit{AD}} \vee \overline{\mathit{DC}}$
is the square facet $\mathit{ABCD}$.
For the complementary vertices $C$ and $E$, the join
$C \vee E$ is the entire cube (since $C$ and $E$ do not coinhabit any
lower-dimensional face).

Our main algorithmic result is the following:

\begin{theorem} \label{t-precomp}
    Consider any polytope
    $P = \{\mathbf{x}\in\R^n\,|\allowbreak\,A\mathbf{x}=\mathbf{b}\allowbreak
    \ \mathrm{and}\ \mathbf{x}\geq 0\}$,
    and suppose we have a list of all vertices of $P$,
    with $V$ vertices in total.
    Then, after a precomputation step that requires
    $O(n^2V + nV^2)$ time and $O(nV^2)$ space:
    \begin{itemize}
        \item we know immediately the dimension of $P$ and
        whether or not $P$ is simple;
        \item if $P$ is simple then, given any two vertices $u,v \in P$,
        we can test whether or not $u$ and $v$ are adjacent in $O(n)$ time;
        \item if $P$ is not simple then, given any two vertices $u,v \in P$,
        we may still be able to identify that $u$ and $v$ are non-adjacent
        in $O(n)$ time.
    \end{itemize}
\end{theorem}

We discuss the importance and implications of this result in
Section~\ref{s-disc}; in the meantime, we devote the remainder of this
section to proving Theorem~\ref{t-precomp}.

Our overall strategy is to use our main theoretical result
(Theorem~\ref{t-comp}) to characterise and identify adjacent vertices.
As a first step, we describe complementary vertices in terms of joins:

\begin{lemma} \label{l-compjoin}
    Let $u$ and $v$ be distinct vertices of a polytope $P$.
    Then $u$ and $v$ are complementary vertices if and only if
    $u \vee v = P$.
\end{lemma}

\begin{proof}
    If $u \vee v \neq P$ then there is some facet
    $F$ for which $u \vee v \subseteq F$; therefore $u,v \in F$,
    and $u$ and $v$ cannot be complementary.

    On the other hand, if $u \vee v = P$ then there is no facet $F$
    for which $u,v \in F$ (otherwise we would have $u \vee v \subseteq F$).
    Therefore $u$ and $v$ are complementary.
\end{proof}

Using this lemma, we can now make a direct link between Theorem~\ref{t-comp}
and adjacency testing in polytopes:

\begin{theorem} \label{t-adj}
    Let $P$ be a simple polytope, and let $F$ be a face of $P$.
    Then $F$ is an edge if and only if
    there is precisely one pair of distinct vertices $u,v \in P$
    for which $F = u \vee v$.

    If $P$ is any polytope (not necessarily simple), then the
    forward direction still holds:
    if $F$ is an edge then there must be
    precisely one pair of distinct vertices $u,v \in P$
    for which $F = u \vee v$.
\end{theorem}

\begin{proof}
    The forward direction is straightforward: let $F$ be an edge of any
    (simple or non-simple) polytope $P$, and let the endpoints of $F$
    be the vertices $u$ and $v$.  It is clear that $F = u \vee v$, and
    because $F$ contains no other vertices it cannot be expressed as the
    join $x \vee y$ for any other pair of distinct vertices $x,y$.

    We now prove the reverse direction in the case where $P$ is a simple
    polytope.  Let $F$ be a face of $P$, and suppose that $F$ is not an
    edge.  If $\dim F < 1$ then $F$ contains at most one vertex, and
    so there can be no pairs of distinct vertices $u,v$ for which
    $F = u \vee v$.

    Otherwise $\dim F > 1$; moreover, since $P$ is a simple polytope
    then the face $F$ is likewise simple when considered as a polytope
    of its own.  Hence we can invoke Theorem~\ref{t-comp} to finish the
    proof.  If $F = u \vee v$ for distinct vertices $u,v$, then
    Lemma~\ref{l-compjoin} shows that $u,v$ are complementary in
    the ``sub-polytope'' $F$.
    By Theorem~\ref{t-comp} there must be another pair of complementary
    vertices $\{u',v'\} \neq \{u,v\}$ in $F$,
    and by Lemma~\ref{l-compjoin} we have $F = u' \vee v'$ as well.
\end{proof}

To make the join operation accessible to algorithms, we describe faces of
polytopes using \emph{zero sets}.
Fukuda and Prodon \cite{fukuda96-doubledesc} define zero sets for points in $P$;
here we extend this concept to arbitrary faces.

\begin{definition}[Zero set]
    Consider any polytope of the form
    $P = \{\mathbf{x}\in\R^n\,|\allowbreak\,A\mathbf{x}=\mathbf{b}
    \ \mathrm{and}\ \mathbf{x}\geq 0\}$,
    and let $F$ be any face of $P$.
    Then the \emph{zero set} of $F$, denoted $Z(F)$, is the set of
    coordinate positions that take the value zero throughout $F$.
    That is, $Z(F) = \{i\,|\,x_i=0\ \mathrm{for\ all}\ \mathbf{x}\in F\}
    \subseteq \{1,2,\ldots,n\}$.
\end{definition}

Zero sets are efficient for computation: each can be stored and
manipulated using a bitmask of size $n$, which for moderate problems
($n \leq 64$) involves just a single machine-native integer and fast
bitwise CPU instructions.
Joins, equality and subface testing all have natural representations
using zero sets:

\begin{lemma} \label{l-zero}
    Let $F$ and $G$ be non-empty faces of the polytope
    $P = \{\mathbf{x}\in\R^n\,|\allowbreak
    \,A\mathbf{x}=\mathbf{b}\ \mathrm{and}\ \mathbf{x}\geq 0\}$.
    Then:
    \begin{itemize}
        \item $F \subseteq G$ if and only if $Z(F) \supseteq Z(G)$;
        \item $F = G$ if and only if $Z(F) = Z(G)$;
        \item the join $F \vee G$ has zero set $Z(F \vee G) = Z(F) \cap Z(G)$.
    \end{itemize}
\end{lemma}

\begin{proof}
    These are all consequences of the fact that every non-empty face
    $F \subseteq P$ is the intersection of $P$
    with all hyperplanes of the form $x_i=0$ where $i \in Z(F)$.
    The details are as follows.

    For each $i=1,\ldots,n$, let
    $H_i$ be the hyperplane
    $H_i=\{\mathbf{x}\in\mathbb{R}^n\,|\,x_i=0\}$.
    It is a standard result that every non-empty face of $P$
    is the intersection of $P$ with some set of hyperplanes $H_i$,
    and conversely that any intersection of $P$ with some set of hyperplanes
    $H_i$ yields a (possibly empty) face of $P$.

    We claim that every non-empty face $F \subseteq P$ satisfies
    $F = P \cap \left(\bigcap_{i \in Z(F)} H_i\right)$;
    that is, $F$ is
    obtained by intersecting $P$ with every $H_i$ for which $i \in Z(F)$.

    To prove this claim, consider the face
    $F' = P \cap \left(\bigcap_{i \in Z(F)} H_i\right)$.
    By definition of $Z(F)$ it is clear that $F \subseteq F'$.
    If $F \neq F'$ then it follows that $F$ is a strict subface of $F'$,
    and so there must be some additional hyperplane $H_j$ with $j \notin Z(F)$
    for which $F \subseteq F' \cap H_j$.  This in turn would imply
    that $j \in Z(F)$, a contradiction.  Therefore $F=F'$.

    We now prove the individual statements of the lemma.
    If $F \subseteq G$
    then it is clear by definition of the zero set
    that $Z(F) \supseteq Z(G)$.  Conversely,
    if $Z(F) \supseteq Z(G)$ then it follows from the claim above that
    $F$ is the intersection of $G$ with zero or more additional
    hyperplanes $H_i$, and so $F \subseteq G$.
    This proves the first statement, and the second statement
    of the lemma now follows immediately.

    We finish with the third statement.
    Since $F,G \subseteq F \vee G$, it follows from the first
    statement that $Z(F \vee G) \subseteq Z(F) \cap Z(G)$.
    Suppose now that $Z(F \vee G) \neq Z(F) \cap Z(G)$; that is,
    there is some $i \in Z(F) \cap Z(G)$ for which
    $i \notin Z(F \vee G)$.  Denote $F' = H_i \cap (F \vee G)$.
    Since $i \notin Z(F \vee G)$ we see that $F'$ is a strict
    subface of $F \vee G$; moreover, since $i \in Z(F) \cap Z(G)$
    we have $F,G \subseteq H_i$ and so
    $F,G \subseteq F'$.  Therefore $F \vee G$ is not the
    smallest-dimensional face containing both $F$ and $G$,
    contradicting the definition of join.
\end{proof}

We now define the main data structure that we build during our
precomputation step in Theorem~\ref{t-precomp}:

\begin{definition}[Join map]
    Consider any polytope of the form
    $P = \{\mathbf{x}\in\R^n\,|\allowbreak\,A\mathbf{x}=\mathbf{b}\allowbreak
    \ \mathrm{and}\ \mathbf{x}\geq 0\}$.
    The \emph{join map} of $P$, denoted $\jm$,
    is a map of the form $\jm\co 2^{\{1,\ldots,n\}} \to \Z_{\geq 0}$;
    that is, $\jm$ maps subsets of $\{1,\ldots,n\}$ to non-negative
    integers.
    For each subset $S \subseteq \{1,\ldots,n\}$,
    we define the image $\jm(S)$ to be the number of pairs
    of distinct vertices $\{u,v\} \in P$ for which $Z(u \vee v) = S$.
\end{definition}

The following result reformulates Theorem~\ref{t-adj} in terms of the join map,
and follows immediately from Theorem~\ref{t-adj} and Lemma~\ref{l-zero}.

\begin{corollary} \label{c-adj-map}
    Let $P$ be a simple polytope, and let $u,v$ be vertices of $P$.
    Then $u$ and $v$ are adjacent if and only if
    $\jm(Z(u) \cap Z(v)) = 1$.

    If $P$ is any polytope (not necessarily simple), then
    the forward direction still holds:
    if $u$ and $v$ are adjacent then we must have
    $\jm(Z(u) \cap Z(v)) = 1$.
\end{corollary}

As a further corollary, the join map can be used to identify precisely
whether or not a polytope is simple:

\begin{corollary} \label{c-join-simple}
    Let $P$ be any polytope of dimension $d$.
    Then $P$ is simple if and only if, for every vertex $u \in P$,
    there are precisely $d$ other vertices $u' \in P$ for which
    $\jm(Z(u) \cap Z(u')) = 1$.
\end{corollary}

\begin{proof}
    If $P$ is simple then, for each vertex $u \in P$, there are
    precisely $d$ other vertices $u' \in P$ adjacent to $u$.
    By Corollary~\ref{c-adj-map} it follows that
    there are precisely $d$ other vertices $u' \in P$ for which
    $\jm(Z(u) \cap Z(u')) = 1$.

    If $P$ is non-simple then there is some vertex $u \in P$ that
    belongs to $>d$ edges, and so there are $>d$ other vertices
    $u' \in P$ adjacent to $u$.
    By Corollary~\ref{c-adj-map},
    we have $\jm(Z(u) \cap Z(u')) = 1$ for each of these adjacent vertices.
\end{proof}

We now show that the join map enjoys many of the properties required for
the complexity bounds in Theorem~\ref{t-precomp}:

\begin{lemma} \label{l-joinmap}
    Consider any polytope
    $P = \{\mathbf{x}\in\R^n\,|\allowbreak\,A\mathbf{x}=\mathbf{b}\allowbreak
    \ \mathrm{and}\ \mathbf{x}\geq 0\}$,
    and suppose we have a list of all vertices of $P$, with
    $V$ vertices in total.
    Then we can construct the join map $\jm$ in $O(nV^2)$ time and
    $O(nV^2)$ space.
    Once it has been constructed, we can compute
    $\jm(S)$ for any set $S \subseteq \{1,\ldots,n\}$
    in $O(n)$ time.
\end{lemma}

\begin{proof}
    We store the join map $\jm$ using a \emph{trie} (also known as a
    \emph{prefix tree}).  This is a binary tree of height $n$.
    Each leaf node (at depth $n$) represents some set
    $S \subseteq \{1,\ldots,n\}$, and stores the corresponding image $\jm(S)$.
    Each intermediate node at depth $k<n$ supports two children:
    a ``left child'' beneath which every set $S$ has $k+1 \notin S$,
    and a ``right child'' beneath which every set $S$ has $k+1 \in S$.

    We optimise our trie by only storing leaf nodes for sets $S$ with
    $\jm(S) \geq 1$, and only storing intermediate nodes that have
    such leaves beneath them.
    Figure~\ref{fig-trie} illustrates the complete trie
    for an example map with $n=3$.

    \begin{figure}[htb]
        \centering
        {\small Join map:\quad
        $\begin{array}{l|c|c|c|c|c|c|c|c}
            S & \,\{\,\}\, & \,\{1\}\, & \,\{2\}\, & \,\{3\}\, &
                \,\{1,2\}\, & \,\{1,3\}\, &
                \,\{2,3\}\, & \,\{1,2,3\}\, \\
            \hline
            \jm(S) & 0 & 0 & 7 & 5 & 4 & 0 & 1 & 1
        \end{array}$}
        \bigskip \\
        \includegraphics[scale=1.0]{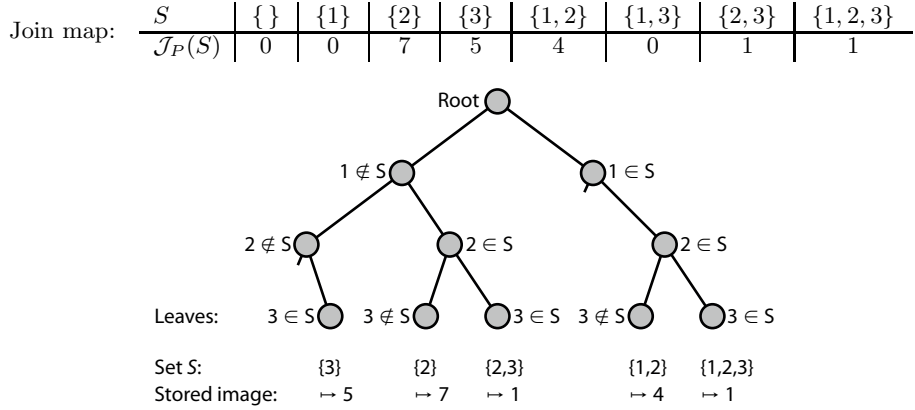}
        \caption{An example of a trie representing a join map for $n=3$}
        \label{fig-trie}
    \end{figure}

    Tries are fast to use: for any set $S \subseteq \{1,\ldots,n\}$,
    the operations of looking up the value of $\jm(S)$,
    inserting a new value of $\jm(S)$,
    or updating an existing value of $\jm(S)$ are all $O(n)$,
    since each operation requires us to follow a single path
    from the root down to level~$n$ (possibly inserting new nodes as we go).
    For further information on tries in general, see a standard text
    such as \cite{knuth98-art3}.

    It is clear now that we can construct $\jm$ in $O(nV^2)$ time: for each
    of the $\binom{V}{2}$ pairs of vertices $\{u,v\}$, we construct
    the set $S=Z(u \vee v)=Z(u) \cap Z(v)$ in $O(n)$ time
    (just test which coordinate positions are zero in both $u$ and $v$),
    and then perform an $O(n)$ lookup for $S$ in the trie.
    If $S$ is present then we increment $\jm(S)$; otherwise
    we perform an $O(n)$ insertion to store the new value $\jm(S)=1$.

    It is also clear that the trie consumes at most
    $O(nV^2)$ space: because the construction involves $O(V^2)$ insertions
    we have $O(V^2)$ leaf nodes, and since the
    trie has depth $n$ this gives $O(nV^2)$ nodes in total.

    Finally, for any set $S \subseteq \{1,\ldots,n\}$,
    computing $\jm(S)$ involves a simple lookup
    operation in the trie, which again requires $O(n)$ time.
\end{proof}

We are now ready to complete the proof of our main result,
Theorem~\ref{t-precomp}:

\begin{proof}[Proof of Theorem~\ref{t-precomp}]
    The precomputation involves three stages:
    \begin{enumerate}[(i)]
        \item building the join map $\jm$;
        \item computing $\dim P$ using a rank computation;
        \item iterating through the join map to determine whether $P$ is
        simple.
    \end{enumerate}

    By Lemma~\ref{l-joinmap}, stage~(i) requires $O(nV^2)$ time and
    $O(nV^2)$ space.

    To compute the dimension in stage~(ii) we select an arbitrary vertex
    $v_0\in P$, and build a $(V-1) \times n$ matrix $M$ whose rows are
    of the form $v-v_0$ for each vertex $v \neq v_0$.  It follows that
    $\dim P = \rank M$, and using standard Gaussian elimination we can
    compute this rank in $O(n^2 V)$ time and $O(nV)$ space.

    For stage~(iii) we once again scan through all $\binom{V}{2}$
    pairs of vertices $u,v$, compute $Z(u) \cap Z(v)$
    for each in $O(n)$ time, and use an $O(n)$ lookup in $\jm$ to
    test whether $\jm(Z(u) \cap Z(v))=1$.
    We can thereby identify whether or not,
    for each vertex $u \in P$, there are precisely $\dim P$ other
    vertices $u' \in P$ for which $\jm(Z(u) \cap Z(u')) = 1$;
    by Corollary~\ref{c-join-simple} this tells us whether or not
    $P$ is simple.
    The total running time for this stage is $O(nV^2)$.

    Summing the three stages together, we find that our precomputation
    step requires a total of $O(n^2V+nV^2)$ time and $O(nV^2)$ space.

    Once this precomputation is complete, it is clear from stages~(ii)
    and (iii) that we know immediately the dimension of $P$ and
    whether or not $P$ is simple.

    Consider now any two vertices $u,v \in P$.  As before, we can compute
    the set $Z(u) \cap Z(v)$ in $O(n)$ time,
    and using Lemma~\ref{l-joinmap} we can evaluate
    $\jm(Z(u) \cap Z(v))$ in $O(n)$ time.
    Now Corollary~\ref{c-adj-map} tells us what we need to know:
    if $P$ is simple then $u$ and $v$ are adjacent if and only if
    $\jm(Z(u) \cap Z(v)) = 1$,
    and if $P$ is non-simple but $\jm(Z(u) \cap Z(v)) \neq 1$
    then we still identify that $u$ and $v$ are non-adjacent.
\end{proof}

%%%%%%%%%%%%%%%%%%%%%%%%%%%%%%%%%%%%%%%%%%%%%%%%%%%%%%%%%%%%%%%%%%%%%%%%
%
%   Discussion
%
%%%%%%%%%%%%%%%%%%%%%%%%%%%%%%%%%%%%%%%%%%%%%%%%%%%%%%%%%%%%%%%%%%%%%%%%

\section{Discussion} \label{s-disc}

As noted in the introduction, the proof of Theorem~\ref{t-comp} is
reminiscent of the Lemke-Howson algorithm for constructing Nash
equilibria \cite{lemke64-games}.
The Lemke-Howson algorithm operates on a pair of simple polytopes
$P$ and $Q$ (best response polytopes for a bimatrix game),
each with precisely $f = \dim P + \dim Q$ facets labelled
$1,\ldots,f$, and locates
vertices $u\in P$ and $v \in Q$ whose incident facet labels combine
to give the full set $\{1,\ldots,f\}$.
See \cite{avis10-nash,lemke64-games} for details.

The Lemke-Howson algorithm can also be framed
in terms of paths through a graph $\Gamma$, where it can be shown that
these paths yield a 1-factor\-is\-ation of $\Gamma$.
One then obtains the corollary that the number of fully-labelled
vertex pairs is even; in particular, because there is always a ``trivial''
pair $(\mathbf{0},\mathbf{0})$, there must be a second pair
(which gives rise to a Nash equilibrium).

In this paper our setting is less well controlled.
We work with a single polytope $P$,
which means we must avoid transforming the pair of vertices $\{u,v\}$
into a pair of the form $\{u,w\}$, or indeed into the identical pair
$\{v,u\}$.
Moreover, $P$ may have arbitrarily many facets,
which makes the arcs of $\auxgraph(P)$ more difficult to categorise.

In particular we do not obtain any such 1-factorisation or parity results,
as noted in Observation~\ref{o-odd}.  Nevertheless, Theorem~\ref{t-2d} shows
that we can obtain parity results under the stronger restriction of
precisely $2d$ facets, which is in fact the setting for the 
the (now disproved) $d$-step conjecture \cite{kim10-hirsch,klee67-dstep}.

Moving to Theorem~\ref{t-precomp}:
our $O(n)$ time adjacency test is the fastest we can hope for, since vertices
require $\Omega(n)$ space to store (a consequence of the fact
that there may be exponentially many vertices \cite{mcmullen70-ubt}).
In contrast, standard approaches to adjacency testing use either
an $O(nV)$ ``combinatorial test'' (where we search for a third vertex
$w \in u \vee v$), or an $O(n^3)$ ``algebraic test'' (where we use a
rank computation to determine $\dim(u \vee v)$).
See \cite{fukuda96-doubledesc} for details of these standard tests.

If we are testing adjacency for all pairs of vertices, our total running
time including precomputation comes to $O(n^2V+nV^2)$, as opposed to
$O(nV^3)$ or $O(n^3V^2)$ for the combinatorial and algebraic tests
respectively.  This is a significant improvement, given that $V$ may be
exponential in $n$.

The main drawback of our test is that it only guarantees
conclusive results for simple polytopes.
Nevertheless, it remains useful in the general case:
it can detect when a polytope is simple, even if this is not clear
from the initial representation
$\{\mathbf{x}\in\mathbb{R}^n\,|\,A\mathbf{x}=\mathbf{b}\ \mbox{and}\
\mathbf{x}\geq 0\}$,
and even in the non-simple setting it gives
a fast filter for eliminating non-adjacent pairs of vertices.

One application of all-pairs adjacency testing is in studying the
\emph{graph} of a polytope; that is, the graph formed from its
vertices and edges.
Another key application is in the
\emph{vertex enumeration} problem: given a polytope in the form
$\{\mathbf{x}\in\mathbb{R}^n\,|\,A\mathbf{x}=\mathbf{b}\ \mbox{and}\
\mathbf{x}\geq 0\}$, identify all of its vertices.  This is a difficult
problem, and it is still unknown whether there exists an algorithm
polynomial in the combined input and output size.

The two best-known algorithms for vertex enumeration are reverse search
% This mbox stops the cite from being split across two lines.
\mbox{\cite{avis00-revised,avis92-pivot}}
and the double description method
\cite{fukuda96-doubledesc,motzkin53-dd}, each with their own advantages
and drawbacks.  The double description method, which
features in several application areas such as
multiobjective optimisation \cite{benson98-outer} and
low-dimensional topology \cite{burton10-dd}, inductively constructs a
sequence of polytopes by adding constraints one at a time.
Its major bottleneck is in identifying pairs of adjacent vertices in each
intermediate polytope,
and in this setting our fast all-pairs adjacency test can be of significant
practical use.

%%%%%%%%%%%%%%%%%%%%%%%%%%%%%%%%%%%%%%%%%%%%%%%%%%%%%%%%%%%%%%%%%%%%%%%%
%
%   Bibliography
%
%%%%%%%%%%%%%%%%%%%%%%%%%%%%%%%%%%%%%%%%%%%%%%%%%%%%%%%%%%%%%%%%%%%%%%%%

\bibliographystyle{amsplain}
\bibliography{pure}

\end{document}